\documentclass[11pt,draft,reqno]{amsart}

\usepackage{amssymb} 
\usepackage{dsfont} 
\usepackage{mathrsfs} 
\usepackage{stmaryrd} 
\usepackage{color}

\DeclareMathAlphabet{\mathpzc}{OT1}{pzc}{m}{it} 

\newtheorem{Thm}{Theorem}[section]
\newtheorem{Cor}{Corollary}[section]
\newtheorem{Lem}{Lemma}[section]
\newtheorem{Prop}{Proposition}[section]

\newtheorem{Def}{Definition}[section]

\theoremstyle{definition}

\theoremstyle{definition}

\makeatletter
\@addtoreset{equation}{section}
\makeatother

\newcommand\X{\textsl{X}\hspace{0.21ex}} 
\renewcommand\d{\textsl{d}} 
\newcommand\e{\textsl{e}} 
\newcommand{\BR}{{\textsl{BR}\hspace{0.17ex}}} 
\newcommand\loc{\mathrm{loc}} 
\newcommand\norm[2]{\Vert #1\Vert_{#2}} 
\DeclareMathOperator{\V}{V} 
\DeclareMathOperator{\ret}{R} 
\newcommand\F{\mathcal{F}}

\newcommand{\AC}{{\textsl{AC}\hspace{0.17ex}}} 
\newcommand\setmeno{\!\smallsetminus\!} 
\newcommand\function{\longrightarrow} 
\newcommand\indicator{\mathds{1}} 
 \newcommand\en{\mathbb{N}} 
\newcommand\ar{\mathbb{R}} 
\providecommand{\clint}[1]{\hspace{0.045ex}\left[#1\right]} 
\providecommand{\clsxint}[1]{\hspace{0.1ex}\left[#1\right[\hspace{0.15ex}} 
\providecommand{\opint}[1]{\hspace{0.15ex}\left]#1\right[\hspace{0.15ex}} 

\newcommand{\borel}{\mathscr{B}} 
\renewcommand{\L}{{\textsl{L}\hspace{0.17ex}}} 
\newcommand\leb{\mathpzc{L}} 
\DeclareMathOperator{\de}{d \! \hspace{0.2ex}} 
\newcommand{\Step}{{\textsl{St}\hspace{0.17ex}}} 

\newcommand\As{\textsl{A}\hspace{0.21ex}} 
\DeclareMathOperator{\Int}{int} 
\DeclareMathOperator{\Cl}{cl} 
\newcommand\Y{\textsl{Y}\hspace{0.21ex}} 
\renewcommand\S{\textsl{S}\hspace{0.21ex}} 

\DeclareMathOperator{\cont}{Cont}  
\DeclareMathOperator{\discont}{Discont}  
\newcommand{\Czero}{{\textsl{C}\hspace{0.18ex}}} 
\DeclareMathOperator{\Lipcost}{Lip} 
\newcommand{\Lip}{{\textsl{Lip}\hspace{0.15ex}}} 

\newcommand\E{\textsl{E}\hspace{0.21ex}} 

\renewcommand\H{\mathcal{H}} 
\newcommand\duality[2]{\langle #1,#2 \rangle} 
\newcommand{\Conv}{\mathscr{C}} 
\newcommand\K{\mathcal{K}} 
\DeclareMathOperator{\Proj}{Proj} 

\newcommand\hausd{\mathpzc{H}} 
\newcommand\A{\mathcal{A}} 
\newcommand\B{\mathcal{B}} 
\newcommand\C{\mathcal{C}} 
\newcommand\G{\mathcal{G}} 

\newcommand\Z{\mathcal{Z}} 

\newcommand{\BV}{{\textsl{BV}\hspace{0.17ex}}} 
\newcommand{\W}{{\textsl{W}\hspace{0.17ex}}} 

\DeclareMathOperator{\pV}{V} 
\renewcommand\r{\textsl{r}} 
\newcommand\vartot[1]{\!\left\bracevert\! #1 \!\right\bracevert\!} 
\DeclareMathOperator{\D}{D\!} 

\newcommand{\Ctilde}{\widetilde{\mathcal{C}}} 
\newcommand{\yhat}{\hat{y}} 

\newcommand{\Sw}{{\mathsf{S}}} 

\providecommand{\cldxint}[1]{\hspace{0.15ex}\left]#1\right]} 
\renewcommand\l{\textsl{l}} 

\definecolor{blu}{rgb}{0.1,0.1,1}
\def\blu#1{\textcolor{blu}{#1}}

\definecolor{green}{rgb}{0.0, 0.5, 0.0}

\definecolor{marr}{rgb}{0.63, 0.47, 0.35}

\begin{document}


\title[Sweeping processes]{Convex valued geodesics and applications to sweeping processes with bounded retraction}

\author{Vincenzo Recupero}
\thanks{The author is a member of GNAMPA-INdAM}

\dedicatory{ \vspace{3ex}
\large 
Dedicated to Professor Alexander Ioffe}

\address{\textbf{Vincenzo Recupero} \\
        Dipartimento di Scienze Matematiche \\ 
        Politecnico di Torino \\
        Corso Duca degli Abruzzi 24 \\ 
        I-10129 Torino \\ 
        Italy. \newline
        {\rm E-mail address:}
        {\tt vincenzo.recupero@polito.it}}




\begin{abstract}
In this paper we provide a formulation for sweeping processes with arbitrary locally bounded retraction, not necessarily 
left or right continuous. Moreover we provide a proof of the existence and uniqueness of solutions for this formulation which relies on 
the reduction to the $1$-Lipschitz continuous case by using a suitable family of geodesics for the asymmetric Hausdorff-like distance called \emph{excess}.
\end{abstract}


\maketitle


\thispagestyle{empty}


\section{Introduction}

Evolution problems with unilateral constraints play a crucial role in many important mechanical applications. A well
known problem of this kind is the so called \emph{sweeping process} introduced by J.J. Moreau in the first seventies
in the articles \cite{Mor71, Mor72, Mor73, Mor74, Mor77}. The most simple formulation of the sweeping processes is the following. Let $\H$ be a real Hilbert space and let $\C(t)$ be a given nonempty, closed and convex subset of $\H$ depending on time such that the mapping $t \longmapsto \C(t)$ is locally Lipschitz continuous on $\clsxint{0,\infty}$ when the family of closed subsets of $\H$ is endowed with the Hausdorff metric. One has to find a locally Lipschitz continuous function $y : \clsxint{0,\infty} \function \H$ such that
\begin{alignat}{3}
  & y(t) \in \C(t) & \quad & \forall t \in \clsxint{0,\infty}, \label{y in C - Lip - intro} \\
  & -y'(t) \in N_{\C(t)}(y(t)) & \quad & \text{for $\leb^{1}$-a.e. $t \in \clsxint{0,\infty}$}, \label{diff. incl. - Lip - intro} \\
  & y(0) = y_{0}, &  \label{in. cond. - Lip - intro}  
\end{alignat}
$y_0$ being a prescribed point in $\C(0)$. Here $\leb^1$ is the Lebesgue measure, and $N_{\C(t)}$ is the exterior normal cone to $\C(t)$ at $y(t)$ (all the definitions will be recalled in Section \ref{S:Preliminaries}).
Moreau was originally motivated by plasticity and friction dynamics (cf. \cite{Mor74, Mor76b, Mor02}), but now sweeping processes have found applications to nonsmooth mechanics (see, e.g., \cite{Mon93, KunMon00, Pao16}), to economics (cf., e.g., \cite{Hen73, Cor83, FlaHirJou09}), to electrical circuits (see, e.g., 
\cite{AddAdlBroGoe07, BroThi10, AcaBonBro11, AdlHadThi14}), to crowd motion modeling (cf., e.g., 
\cite{MauVen08, MauRouSan10, MauVen11, MauRouSanVen11, MauRouSan14, DimMauSan16}), and to other fields (see, e.g., the references in the recent paper \cite{Thi16}).
The theoretical analysis of problem \eqref{y in C - BV - intro}-\eqref{in. cond. - BV - intro} has been expanded in various directions: the case of $\C$ continuous was first dealt in \cite{Mon93},
and in \cite{Mon84} the application of external forces is also considered;
the nonconvex case has been studied in several papers, e.g. 
\cite{Val88, Val90, Val92, CasDucVal93, CasMon96, ColGon99, Ben00, ColMon03, Thi03, BouThi05, EdmThi06, Thi08, HasJouThi08, BerVen10, SenThi14};  
for stochastic versions see, e.g., \cite{Cas73, Cas76, BerVen11, CasMonRay16}, while periodic solutions can be found in \cite{CasMon95}. The continuous dependence properties of various sweeping problems are investigated, e.g., in 
\cite{Mor77, BroKreSch04, KreVla03, Rec09, KreRoc11, Rec11a, Rec11c, Rec15a, KleRec16, KopRec16},
and the control problems are studied, e.g., in \cite{ColHenHoaMor15, ColHenNguMor16, ColPal16}.

In \cite{Mor77} the formulation \eqref{y in C - Lip - intro}-\eqref{in. cond. - Lip - intro} is extended to the case when the mapping $t \longmapsto \C(t)$ is of locally bounded variation and right continuous in the following natural way: it is proved that there is a unique $y \in \BV^\r_\loc(\clsxint{0,\infty};\H)$, the space of right continuous $\H$-valued functions of locally bounded variation, such that there exist a positive measure $\mu$ and a $\mu$-integrable function $v : \clsxint{0,\infty} \function \H$ satisfying
\begin{alignat}{3}
  & y(t) \in \C(t) & \quad &  \forall t \in \clsxint{0,\infty}, \label{y in C - BV - intro} \\
  & \! \D y = v \mu, & \label{Dy = w mu - intro} \\
  & -v(t) \in N_{\C(t)}(y(t)) & \quad & \text{for $\mu$-a.e. $t \in \clsxint{0,\infty}$}, \label{diff. incl. - BV - intro} \\
  & y(a) = y_{0}, &  \label{in. cond. - BV - intro}
\end{alignat}
where $\D y$ denotes distributional derivative of $y$, which is a measure since $y$ is of bounded variation.

When $t$ is a jump point for $\C$, the normality conditions
\eqref{Dy = w mu - intro}-\eqref{diff. incl. - BV - intro} means that  $y(t) = y(t+) = \Proj_{\C(t)}(y(t-))$, where 
$\Proj$ is the classical projection operator, thus $y$ jumps from $\C(t-)$ to $\C(t) = \C(t+)$ along the shortest path which allows to satisfy the constraint \eqref{y in C - BV - intro}.

The existence and uniqueness proof for problem \eqref{y in C - BV - intro}--\eqref{in. cond. - BV - intro} relies upon
an implicit approximation scheme that is usually called "catching up algorithm" in view of the geometric meaning of the 
projection on $\C(t+)$. Actually in \cite{Mor77} the moving set $\C$ is assumed to be only with locally bounded right continuous retraction, rather than with locally bounded variation (the notion of retraction is recalled in \eqref{retraction}), i.e. the Hausdorff distance is replaced by the asymmetric distance $\e(\A,\B) := \sup_{x \in \A} \d(x,\B)$, called \emph{excess of $\A$ over $\B$}, with $\A$, $\B$ nonempty closed convex sets in $\H$. When $\C$ 
does not enjoy the right continuity property with respect to the excess $\e$, a formulation involving measures like \eqref{y in C - BV - intro}--\eqref{in. cond. - BV - intro} is not possible since at a jump point $t$ one has $y(t) = \Proj_{\C(t)}(y(t-))$ and
$y(t+) = \Proj_{\C(t+)}(y(t))$, thus in \cite{Mor77} a weak solution is defined as the uniform limit
of the approximate solutions defined by the catching up algorithm.

The purpose of the present paper is twofold. First we want to introduce a sort of ``differential measure'' formulation
similar to \eqref{y in C - BV - intro}--\eqref{in. cond. - BV - intro} also in the case when $\C$ is with locally bounded 
retraction but not necessarily left or right continuous. On the other hand we provide an existence proof for our 
formulation which reduces the problem to the case when $\C$ has locally Lipschitz retraction. 
In order to perform this proof we reparametrize the convex valued curve $\C$ by the ``arc length'' $\ell_\C(t)$ with respect to the excess $\e$ (for simplicity let us assume now that there are no intervals where $\ell_\C$ is constant) so that we can write $\C(t) = \Ctilde(\ell_\C(t))$ for every $t$ and for a suitable $\e$-Lipschitz reparametrization $\Ctilde(\sigma)$.
The problem of reducing to the Lipschitz case is that $\Ctilde$ is only defined in the image of $\ell_\C$, therefore we have to fill in the jumps of $\C$ in an appropriate way in order to have $\Ctilde$ defined on the whole $\clsxint{0,\infty}$.
In this way we can take the Lispchitz solution of the sweeping process associated to $\Ctilde$, and then throw away
the jumps to get the discontinuous solution related to $\C$.
In the past papers \cite{Rec11a,KreRec14a,KreRec14b,Rec16a} we have shown that the choice of the paths connecting these jumps is nontrivial and an arbitrary connection can produce the wrong solution. In the paper 
\cite{Rec16a} we found the proper family of geodesic connecting the jumps in the case when $\C$ is of bounded variation with respect to the Hausdorff metric. In the present paper we exploit a family of convex valued curves which are geodesics with respect to the excess $\e$. In a certain sense this procedure allows us to say that every discontinuous sweeping process with locally bounded retraction ``is indeed'' a Lipschitz continuous one. In this way we also generalize a result of \cite{RecSan18} where a kind of ``differential measure'' formulation is given for sweeping processes with bounded variation with respect to the Hausdorff metric. In that paper we also allowed the 
behaviour on jumps to have a more general behaviour, therefore it is natural to wonder if any of these general behaviors can be obtained by a reparametrization of $\C$ using suitable curves connecting the jumps.

The paper is organized as follows. In the next section we present some preliminaries and in Sections 3 we state the main results of the paper. In Sections 4 and 5 we review the proof of the Lipschitz continuous case and we provide 
an integral formulation of the Lipschitz continuous sweeping processes. In Section 6 we define and study the class
of convex valued geodesics needed in the proof of our main result which is presented in the last Section 7.


\section{Preliminaries}\label{S:Preliminaries}

Let us now recall the main notations and definitions used throughout in the paper. The set of integers greater than or equal to $1$ will be denoted by $\en$. Given an interval $I$ of the real line $\ar$ and the family $\borel(I)$ of Borel subsets of $I$, if $\mu : \borel(I) \function \clint{0,\infty}$ is a measure, $p \in \clint{1,\infty}$, and if 
$\E$ is a Banach space, then the space of $\E$-valued functions which are $p$-integrable with respect to $\mu$ will be denoted by $\L^p(I, \mu; \E)$ or by $\L^p(\mu; \E)$. We do not identify two functions which are equal $\mu$-almost everywhere ($\mu$-a.e.). The one dimensional Lebesgue measure is denoted by $\leb^1$ and for the theory of integration of vector valued functions we refer, e.g., to \cite[Chapter VI]{Lan93}.

\subsection{Functions of bounded variation}

In this subsection we assume that 
\begin{equation}\label{X complete metric space}
  \text{$(\X,\d)$ is an extended complete metric space},
\end{equation}
i.e. $\X$ is a set and $\d : \X \times \X \function \clint{0,\infty}$ satisfies 
the usual axioms of a distance, but may take on the value $\infty$. The notion of completeness remains unchanged. 
The general topological notions of interior, closure and boundary of a subset $\Y \subseteq \X$ will be respectively denoted by $\Int(\Y)$, $\Cl(\Y)$ and $\partial \Y$. We also set $\d(x,\As) := \inf_{a \in \As} \d(x,a)$. If $(\Y,\d_\Y)$ is a metric space then the continuity set of a function $f : \Y \function \X$ is denoted by $\cont(f)$, while 
$\discont(f) := \Y \setmeno \cont(f)$. For $\S \subseteq \Y$ we write 
$\Lipcost(f,\S) := \sup\{\d(f(s),f(t))/\d_\Y(t,s)\ :\ s, t \in \S,\ s \neq t\}$, $\Lipcost(f) := \Lipcost(f,\Y)$, the Lipschitz constant of $f$, and $\Lip(\Y;\X) := \{f : \Y \function \X\ :\ \Lipcost(f) < \infty\}$, the set of $\X$-valued Lipschitz continuous functions on $\Y$. As usual $\Lip_\loc(\Y;\X) := \{f \in \X^I\ :\ \Lipcost(f,\S) < \infty \ \text{$\forall \S$  compact in $\Y$}\}$.

\begin{Def}
Given an interval $I \subseteq \ar$, a function $f : I \function \X$, and a subinterval $J \subseteq I$, the 
\emph{(pointwise) variation of $f$ on $J$} is defined by
\begin{equation}\notag
  \pV(f,J) := 
  \sup\left\{
           \sum_{j=1}^{m} \d(f(t_{j-1}),f(t_{j}))\ :\ m \in \en,\ t_{j} \in J\ \forall j,\ t_{0} < \cdots < t_{m} 
         \right\}.
\end{equation}
If $\pV(f,I) < \infty$ we say that \emph{$f$ is of bounded variation on $I$} and we set 
$\BV(I;\X) := \{f : I \function \X\ :\ \pV(f,I) < \infty\}$ and $
\BV_\loc(I;\X) := \{f : I \function \X\ :\ \pV(f,J) < \infty\ \text{$\forall J$  \emph{compact} in $I$}\}$.
\end{Def}

It is well known that the completeness of $\X$ implies that every $f \in \BV(I;\X)$ admits one-sided limits $f(t-), f(t+)$ at every point $t \in I$, with the convention that $f(\inf I-) := f(\inf I)$ if $\inf I \in I$, and $f(\sup I+) := f(\sup I)$ if $\sup I \in I$. Moreover $\discont(f)$ is at most countable. We set $\BV^\l(I;\X) := \{f \in \BV(I;\X)\ :\ f(t-) = f(t) \quad \forall t \in I\}$, 
$\BV^\r(I;\X) := \{f \in \BV(I;\X)\ :\ f(t) = f(t+) \quad \forall t \in I\}$. Accordingly we set 
$\BV_\loc^\l(I;\X) := \{u \in \BV_\loc(I;\X)\ :\ u(t-) = u(t) \quad \forall t \in I\}$ and
$\BV_\loc^\r := \{u \in \BV_\loc(I;\X)\ :\ u(t) = u(t+) \quad \forall t \in I\}$ so that
$\Lip_\loc(I;\X) \subseteq \BV_\loc(I;\X)$.


\subsection{Convex sets in Hilbert spaces}

Throughout the remainder of the paper we assume that
\begin{equation}\label{H-prel}
\begin{cases}
  \text{$\H$ is a real Hilbert space with inner product $(x,y) \longmapsto \duality{x}{y}$}, \\
  \norm{x}{} := \duality{x}{x}^{1/2},
\end{cases}
\end{equation}
and we endow $\H$ with the natural metric defined by $\d(x,y) := \norm{x-y}{}$, $x, y \in \H$. 
We set
\[
  D_r := \{x \in \H\ :\ \norm{x}{} \le r\}, \qquad r > 0
\]
and
\begin{equation}\notag
  \Conv_\H := \{\K \subseteq \H\ :\ \K \ \text{nonempty, closed and convex} \}.
\end{equation}
If  $\K \in \Conv_\H$ and $x \in \H$, then $\Proj_{\K}(x)$ is the projection on $\K$, i.e. $y = \Proj_\K(x)$ is the unique 
point such that $d(x,\K) = \norm{x-y}{}$, and it is the only element $y \in \H$ such that the two conditions
\begin{equation}\notag
  y \in \K, \quad \duality{x - y}{v - y} \le 0 \qquad \forall v \in \K,
\end{equation}
hold. If $\K \in \Conv_\H$ and $x \in \K$, then $N_\K(x)$ denotes the \emph{(exterior) normal cone of $\K$ at $x$}:
\begin{equation}\label{normal cone}
  N_\K(x) := \{u \in \H\ :\ \duality{v - x}{u} \le 0\ \forall v \in \K\} = \Proj_\K^{-1}(x) - x.
\end{equation}
From the previous definition in \eqref{normal cone} it follows that the multivalued mapping $x \longmapsto N_\K(x)$ is \emph{monotone}, i.e. 
$\duality{u_1 - u_2}{x_1 - x_2} \ge 0$ whenever $x_j \in \K$, $u_j \in N_\K(x_j)$, $j = 1, 2$ (see, e.g., 
\cite[Exemple 2.8.2, p.46]{Bre73}). We endow the set $\Conv_{\H}$ with the so called \emph{excess}. Here we recall the definition.

\begin{Def}
The \emph{excess} $\e: \Conv_{\H} \times \Conv_{\H} \function \clint{0,\infty}$ is defined by
\begin{equation}\notag
  \e(\A,\B) := \sup_{a \in \A} \d(a,\B), \qquad \A, \B \in \Conv_{\H},
\end{equation}
which is also called the \emph{excess of $\A$ over $\B$.}
\end{Def}

The following facts are well known (see \cite{Mor74b}):
\begin{equation}
  \e(\A,\A) =0 \qquad \forall \A \in \Conv_\H,
\end{equation}
\begin{equation}
  \e(\A,\B) \le \e(\A,\C) + \e(\C,\B) \qquad \forall \A, \B, \C \in \Conv_\H.
\end{equation}
Moreover we have that
\begin{equation}
  \e(\A,\B) = \inf\{\rho > 0 \ :\ \A \subseteq \B + D_\rho\} \qquad \forall \A, \B \in \Conv_\H.
\end{equation}
\begin{Def}
Given an interval $I \subseteq \ar$, a function $\C : I \function \Conv_\H$, and a subinterval $J \subseteq I$, the 
\emph{retraction of $\C$ on $J$} is defined by
\begin{equation}\label{retraction}
  \ret(\C,J) := 
  \sup\left\{
           \sum_{j=1}^{m} \e(\C(t_{j-1}),\C(t_{j}))\ :\ m \in \en,\ t_{j} \in J\ \forall j,\ t_{0} < \cdots < t_{m} 
         \right\}.
\end{equation}
If $\ret(\C,I) < \infty$ we say that \emph{$\C$ is of bounded retraction on $I$} and we set 
$\BR(I;\Conv_\H) := \{\C : I \function \Conv_\H\ :\ \ret(\C,I) < \infty\}$. We also set 
$\BR_\loc(I;\Conv_\H) := \{\C : I \function \Conv_\H\ :\ \ret(\C,J) < \infty\ \text{$\forall J$ compact in $I$}\}$.
\end{Def}
For every $\C \in \BR_\loc(I;\Conv_\H)$ we will define $\ell_\C : \clsxint{0,\infty} \function \clsxint{0,\infty}$ by
\begin{equation}
  \ell_\C(t) := \ret(\C;\clint{0,t}), \qquad t \ge 0.
\end{equation}
The function $\ell_\C$ is a sort of arc length with respect to the excess $\e$ and is an increasing function such that
$\ell_\C(0) = 0$. 
It is well known (cf. \cite{Mor74b}) that if $\C \in \BR_\loc(I;\Conv_\H)$ then for every $t \in I$ we have
\begin{equation}
  \C(t+) := \liminf_{s \to t+} \C(s)  := \{x \in \H\ :\ \lim_{s \to t+} \d(x,\C(s)) = 0\} \in \Conv_\H,
\end{equation}
\begin{equation}
  \C(t-) := \liminf_{s \to t-} \C(s) := \{x \in \H\ :\ \lim_{s \to t-} \d(x,\C(s)) = 0\} \in \Conv_\H
\end{equation}
and
\begin{equation}
  \e(\C(t),\C(t+)) = \lim_{s \to t+}\e(\C(t),\C(s)) = \ell_\C(t+) - \ell_\C(t),
\end{equation}
\begin{equation}
  \e(\C(t-),\C(t)) = \lim_{s \to t-}\e(\C(s),\C(t)) = \ell_\C(t) - \ell_\C(t-),
\end{equation}
therefore
\begin{equation}
  \lim_{s \to t+}\e(\C(t),\C(s)) = 0 \ \Longleftrightarrow\ \C(t) \subseteq \C(t+)
\end{equation}
and
\begin{equation}
  \lim_{s \to t-}\e(\C(s),\C(t)) = 0 \ \Longleftrightarrow\ \C(t-) \subseteq \C(t).
\end{equation}
Thus for $\C \in \BR_\loc(I;\Conv_\H)$ we set $\cont(\C) := \{t \in I\ :\ \e(\C(t),\C(t+)) = \e(\C(t-),\C(t)) = 0\}$
and $\discont(\C) := I \setmeno \cont(\C)$. Moreover we set 
$\BR^\r(I;\Conv_\H) := \{\C \in \BR(I;\Conv_\H)\ :\ \e(\C(t),\C(t+)) = 0\ \forall t \in I\}$ and 
$\BR^\l(I;\Conv_\H) := \{\C \in \BR(I;\Conv_\H)\ :\ \e(\C(t-),\C(t)) = 0\ \forall t \in I\}$. Finally 
$\BR^\r_\loc(I;\Conv_\H) := \{\C \in \BR^\loc(I;\Conv_\H)\ :\ \e(\C(t),\C(t+)) = 0\ \forall t \in I\}$ and
$\BR^\l_\loc(I;\Conv_\H) := \{\C \in \BR^\loc(I;\Conv_\H)\ :\ \e(\C(t-),\C(t)) = 0\ \forall t \in I\}$.


\subsection{Differential measures}

We recall that a \emph{$\H$-valued measure on $I$} is a map $\mu : \borel(I) \function \H$ such that 
$\mu(\bigcup_{n=1}^{\infty} B_{n})$ $=$ $\sum_{n = 1}^{\infty} \mu(B_{n})$ whenever $(B_{n})$ is a sequence of mutually disjoint sets in $\borel(I)$. The \emph{total variation of $\mu$} is the positive measure 
$\vartot{\mu} : \borel(I) \function \clint{0,\infty}$ defined by
\begin{align}\label{tot var measure}
  \vartot{\mu}(B)
  := \sup\left\{\sum_{n = 1}^{\infty} \norm{\mu(B_{n})}{}\ :\ 
                 B = \bigcup_{n=1}^{\infty} B_{n},\ B_{n} \in \borel(I),\ 
                 B_{h} \cap B_{k} = \varnothing \text{ if } h \neq k\right\}. \notag
\end{align}
The vector measure $\mu$ is said to be \emph{with bounded variation} if $\vartot{\mu}(I) < \infty$. In this case the equality $\norm{\mu}{} := \vartot{\mu}(I)$ defines a norm on the space of measures with bounded variation.
Finally we say that $\mu$ is \emph{with local bounded variation} if $\vartot{\mu}(J) < \infty$ for every interval $J$ compact in $I$ (see, e.g. \cite[Chapter I, Section  3]{Din67}). 

If $\nu : \borel(I) \function \clint{0,\infty}$ is a positive bounded Borel measure and if $g \in \L^1(I,\nu;\H)$, then $g\nu$ will denote the vector measure defined by $g\nu(B) := \int_B g\de \nu$ for every $B \in \borel(I)$. In this case 
$\vartot{g\nu}(B) = \int_B \norm{g(t)}{}\de \nu$ for every $B \in \borel(I)$ (see \cite[Proposition 10, p. 174]{Din67}). Moreover a vector measure $\mu$ is called \emph{$\nu$-absolutely continuous} if $\mu(B) = 0$ whenever 
$B \in \borel(I)$ and $\nu(B) = 0$. 

Assume that $\mu : \borel(I) \function \H$ is a vector measure with bounded variation and let $f : I \function \H$ and 
$\phi : I \function \ar$ be two \emph{step maps with respect to $\mu$}, i.e. there exist $f_{1}, \ldots, f_{m} \in \H$, 
$\phi_{1}, \ldots, \phi_{m} \in \H$ and $A_{1}, \ldots, A_{m} \in \borel(I)$ mutually disjoint such that 
$\vartot{\mu}(A_{j}) < \infty$ for every $j$ and $f = \sum_{j=1}^{m} \indicator_{A_{j}} f_{j},$, 
$\phi = \sum_{j=1}^{m} \indicator_{A_{j}} \phi_{j},$ where $\indicator_{S} $ is the characteristic function of a set $S$, i.e. 
$\indicator_{S}(x) := 1$ if $x \in S$ and $\indicator_{S}(x) := 0$ if $x \not\in S$. For such step functions we define 
$\int_{I} \duality{f}{\mu} := \sum_{j=1}^{m} \duality{f_{j}}{\mu(A_{j})} \in \ar$ and
$\int_{I} \phi \de \mu := \sum_{j=1}^{m} \phi_{j} \mu(A_{j}) \in \H$. If $\Step(\vartot{\mu};\H)$ (resp. $\Step(\vartot{\mu})$) is the set of $\H$-valued (resp. real valued) step maps with respect to $\mu$, then the maps
$\Step(\vartot{\mu};\H)$ $\function$ $\H : f \longmapsto \int_{I} \duality{f}{\mu}$ and
$\Step(\vartot{\mu})$ $\function$ $\H : \phi \longmapsto \int_{I} \phi \de \mu$ 
are linear and continuous when $\Step(\vartot{\mu};\H)$ and $\Step(\vartot{\mu})$ are endowed with the 
$\L^{1}$-seminorms $\norm{f}{\L^{1}(\vartot{\mu};\H)} := \int_I \norm{f}{} \de \vartot{\mu}$ and
$\norm{\phi}{\L^{1}(\vartot{\mu})} := \int_I |\phi| \de \vartot{\mu}$. Therefore they admit unique continuous extensions 
$\mathsf{I}_{\mu} : \L^{1}(\vartot{\mu};\H) \function \ar$ and 
$\mathsf{J}_{\mu} : \L^{1}(\vartot{\mu}) \function \H$,
and we set 
\[
  \int_{I} \duality{f}{\de \mu} := \mathsf{I}_{\mu}(f), \quad
  \int_{I} \phi \mu := \mathsf{J}_{\mu}(\phi),
  \qquad f \in \L^{1}(\vartot{\mu};\H),\quad \phi \in \L^{1}(\vartot{\mu}).
\]

If $\nu$ is bounded positive measure and $g \in \L^{1}(\nu;\H)$, arguing first on step functions, and then taking limits, it is easy to check that 
\begin{equation}\label{f de gnu =fg de nu}
  \int_I\duality{f}{\de(g\nu)} = \int_I \duality{f}{g}\de \nu \qquad \forall f \in \L^{\infty}(\mu;\H).
\end{equation} 
The following results (cf., e.g., \cite[Section III.17.2-3, pp. 358-362]{Din67}) provide a connection between functions with bounded variation and vector measures which will be implicitly used in the paper.

\begin{Thm}\label{existence of Stietjes measure}
For every $f \in \BV_\loc(I;\H)$ there exists a unique vector measure of local bounded variation $\mu_{f} : \borel(I) \function \H$ such that 
\begin{align}
  \mu_{f}(\opint{c,d}) = f(d-) - f(c+), \qquad \mu_{f}(\clint{c,d}) = f(d+) - f(c-), \notag \\ 
  \mu_{f}(\clsxint{c,d}) = f(d-) - f(c-), \qquad \mu_{f}(\cldxint{c,d}) = f(d+) - f(c+). \notag 
\end{align}
whenever $c < d$ and the left hand side of each equality makes sense. Conversely, if $\mu : \borel(I) \function \H$ is a vector measure with local bounded variation, and if $f_{\mu} : I \function \H$ is defined by 
$f_{\mu}(t) := \mu(\clsxint{\inf I,t} \cap I)$, then $f_{\mu} \in \BV_\loc(I;\H)$ and $\mu_{f_{\mu}} = \mu$.
\end{Thm}

\begin{Prop}
Let $f  \in \BV_\loc(I;\H)$, let $g : I \function \H$ be defined by $g(t) := f(t-)$, for $t \in \Int(I)$, and by $g(t) := f(t)$, if 
$t \in \partial I$, and let $V_{g} : I \function \ar$ be defined by $V_{g}(t) := \pV(g, \clint{\inf I,t} \cap I)$. Then  
$\mu_{g} = \mu_{f}$ and $\vartot{\mu_{f}} = \mu_{V_{g}} = \pV(g,I)$.
\end{Prop}

The measure $\mu_{f}$ is called \emph{Lebesgue-Stieltjes measure} or \emph{differential measure} of $f$. Let us see the connection with the distributional derivative. If $f \in \BV_\loc(I;\H)$ and if $\overline{f}: \ar \function \H$ is defined by
\begin{equation}\label{extension to R}
  \overline{f}(t) :=
  \begin{cases}
    f(t) 	& \text{if $t \in I$} \\
    f(\inf I)	& \text{if $\inf I \in \ar$, $t \not\in I$, $t \le \inf I$} \\
    f(\sup I)	& \text{if $\sup I \in \ar$, $t \not\in I$, $t \ge \sup I$}
  \end{cases},
\end{equation}
then, as in the scalar case, it turns out (cf. \cite[Section 2]{Rec11a}) that $\mu_{f}(B) = \D \overline{f}(B)$ for every 
$B \in \borel(\ar)$, where $\D\overline{f}$ is the distributional derivative of $\overline{f}$, i.e.
\[
  - \int_\ar \varphi'(t) \overline{f}(t) \de t = \int_{\ar} \varphi \de \D \overline{f} 
  \qquad \forall \varphi \in \Czero_{c}^{1}(\ar;\ar),
\]
$\Czero_{c}^{1}(\ar;\ar)$ being the space of real continuously differentiable functions on $\ar$ with compact support.
Observe that $\D \overline{f}$ is concentrated on $I$: $\D \overline{f}(B) = \mu_f(B \cap I)$ for every $B \in \borel(I)$, hence in the remainder of the paper, if $f \in \BV_\loc(I,\H)$ then we will simply write
\begin{equation}
  \D f := \D\overline{f} = \mu_f, \qquad f \in \BV(I;\H),
\end{equation}
and from the previous discussion it follows that 
\begin{equation}
  \norm{\D f}{} = \vartot{\D f}(I) = \norm{\mu_f}{}  = \pV(f,I) \qquad \forall f \in \BV^\r(I;\H).
\end{equation}
If $I$ is bounded and $p \in \clint{1,\infty}$, then the classical Sobolev space $\W^{1,p}(I;\H)$ consists of those functions $f \in \Czero(I;\H)$ such that $\D f = g\leb^1$ for some $g \in \L^p(I;\H)$ and we have  $\W^{1,p}(I;\H) = \AC^p(I;\H)$. Let us also recall that if $f \in \W^{1,1}(I;\H)$ then {}{ the derivative $f'(t)$  exists} for $\leb^1$-a.e. in $t \in I$, 
$\D f = f' \leb^1$, and $\V(f,I) = \int_I\norm{f'(t)}{}\de t$ (cf., e.g. \cite[Appendix]{Bre73}).

In \cite[Lemma 6.4 and Theorem 6.1]{Rec11a} it is proved that

\begin{Prop}\label{P:BV chain rule}
Assume that $I, J \subseteq \ar$ are intervals and that $h : I \function J$ is nondecreasing.
\begin{itemize}
\item[(i)]
  $\D h(h^{-1}(B)) = \leb^{1}(B)$ for every $B \in \borel(h(\cont(h)))$.
\item[(ii)]
 If $f \in \Lip(J;\H)$ and $g : I \function \H$ is defined by
\begin{equation}\notag
  g(t) := 
  \begin{cases}
    f'(h(t)) & \text{if $t \in \cont(h)$} \\
    \ \\
    \dfrac{f(h(t+)) - f(h(t-))}{h(t+) - h(t-)} & \text{if $t \in \discont(h)$}
  \end{cases},
\end{equation}
then $f \circ h \in \BV(I;\H)$ and $\D\ \!(f \circ h) = g \D h$. This result holds even if $f'$ is replaced by any of its 
$\leb^{1}$-representatives.
\end{itemize}
\end{Prop}


\section{Main results}

In this section we state the main theorem of the present paper.

\begin{Thm}\label{main thm}
Assume that $\C \in \BR_\loc(\clsxint{0,\infty};\Conv_\H)$ and $y_0 \in \H$. Then there exists a unique 
$y \in \BV_\loc(\clsxint{0,\infty};\H)$ such that there exists a Borel measure 
$\mu : \borel(\clsxint{0,\infty}) \function \clint{0,\infty}$ and a function $v \in \L^1_\loc(\mu;\H)$ such that
\begin{align}
  & y(t) \in \C(t), \label{constr. arbBV sweep} \\
  & \D y = v \mu, \label{Dy = vmu}\\
  & -v(t) \in N_{\C(t)}(y(t)) \qquad \text{for $\mu$-a.e. $t \in \cont(\C)$}, \label{diff inclu}\\
  & y(t) = \Proj_{\C(t)}(y(t-)), \qquad
   y(t+) = \Proj_{\C(t+)}(y(t)) \qquad \forall t \in \discont(\C) \!\ \setmeno \!\ \{0\} \label{jump cond}\\
  & y(0) = \Proj_{\C(0)}(y_0), \quad y(0+) = \Proj_{\C(0+)(y(0))}. \label{i.c. arbBV sweep}
\end{align}
Moreover $y$ is left continuous (respectively: right continuous) at $t \ge 0$ if and only if $u$ is left continuous (respectively: right continuous) at $t \ge 0$.
\end{Thm}

The following result shows that in the right continuous case the conditions 
\eqref{diff inclu}--\eqref{i.c. arbBV sweep} reduce to 
\begin{align}\label{single integral condition}
  & -v(t) \in N_{\C(t)}(y(t)) \qquad \text{for $\mu$-a.e. $t \in \opint{0,\infty}$}, \\
  & y(0) = \Proj_{\C(0)}(y_0).
\end{align}

\begin{Thm}\label{main-right continuous case}
Assume that $\C \in \BR_\loc^\r(\clsxint{0,\infty};\Conv_\H)$ and $y_0 \in \H$. Then there exists a unique 
$y \in \BV_\loc^\r(\clsxint{0,\infty};\H)$ such that there exists a Borel measure 
$\mu : \borel(\clsxint{0,\infty}) \function \clint{0,\infty}$ and a function $v \in \L^1_\loc(\mu;\H)$ such that
\begin{align}
  & y(t) \in \C(t), \label{constr. arbBV sweep-r} \\
  & \D y = v \mu, \\
  & -v(t) \in N_{\C(t)}(y(t)) \qquad \text{for $\mu$-a.e. $t \in \clsxint{0,\infty}$}, \label{diff inclu-r}\\
  & y(0) = \Proj_{\C(0)}(y_0). \label{i.c. arbBV sweep-r}
\end{align}
\end{Thm}

An analogous result holds in the left continuous case:

\begin{Thm}\label{main-left continuous case}
Assume that $\C \in \BR_\loc^\l(\clsxint{0,\infty};\Conv_\H)$ and $y_0 \in \H$. Then there exists a unique 
$y \in \BV_\loc^\l(\clsxint{0,\infty};\H)$ such that there exists a Borel measure 
$\mu : \borel(\clsxint{0,\infty}) \function \clint{0,\infty}$ and a function $v \in \L^1_\loc(\mu;\H)$ such that
\begin{align}
  & y(t) \in \C(t), \label{constr. arbBV sweep-l} \\
  & \D y = v \mu, \\
  & -v(t) \in N_{\C(t+)}(y(t+)) \qquad \text{for $\mu$-a.e. $t \in \clsxint{0,\infty}$}, \label{diff inclu-l}\\
  & y(0) = \Proj_{\C(0)}(y_0). \label{i.c. arbBV sweep-l}
\end{align}
\end{Thm}


\section{Uniqueness and integral formulations}\label{Uniqueness}

We start by proving the uniqueness of solutions to \eqref{constr. arbBV sweep}--\eqref{i.c. arbBV sweep} in
Theorem \ref{main thm}.

\begin{Lem}\label{uniqueness}
Assume that $\C \in \BR_\loc(\clsxint{0,\infty};\Conv_\H)$ and $y_0 \in \H$. Then there exists a unique 
$y \in \BV_\loc(\clsxint{0,\infty};\H)$ such that there exists a Borel measure 
$\mu : \borel(\clsxint{0,\infty}) \function \clint{0,\infty}$ and a function $v \in \L^1_\loc(\mu;\H)$ satisfying
\eqref{constr. arbBV sweep}--\eqref{i.c. arbBV sweep}.
\end{Lem}

\begin{proof}
Assume by contradiction that there are two solutions $y_1$ and $y_2$.
If $B \in \borel(\clsxint{0,\infty})$, then by \cite[Proposition 2]{Mor76} and by the monotonicity of the normal cone
we get
\begin{align}
  \int_{B \cap \cont(\C)} \de \D\!\ (\norm{y_1(\cdot) - y_2(\cdot)}{}^2) 
    & \le 2 \int_{B \cap \cont(\C)} \duality{y_1 - y_2}{\de \D\!\ (y_1 - y_2)}  \notag \\
    & \le 2 \int_{\cont(\C)} \duality{y_1 - y_2}{\de \D\!\ (y_1 - y_2)} \le 0, \notag
\end{align}
while if $t \in \discont(\C)$, from \eqref{jump cond} we infer that
\begin{align}
  & \D\!\ (\norm{y_1(\cdot) - y_2(\cdot)}{}^2)(\{t\}) \notag \\
    & =   \norm{y_1(t+) - y_2(t+)}{}^2 - \norm{y_1(t-) - y_2(t-)}{}^2 \notag \\
    & =   \norm{\Proj_{\C(t+)}(y_1(t)) - \Proj_{\C(t+)}(y_2(t))}{}^2 - \norm{y_1(t-) - y_2(t-)}{}^2 \notag \\ 
    & \le \norm{y_1(t) - y_2(t)}{}^2 - \norm{y_1(t-) - y_2(t-)}{}^2 \notag \\
    & = \norm{\Proj_{\C(t)}(y_1(t-)) - \Proj_{\C(t)}(y_2(t-))}{}^2 - \norm{y_1(t-) - y_2(t-)}{}^2 \notag \\
    & \le \norm{y_1(t-) - y_2(t-)}{}^2 - \norm{y_1(t-) - y_2(t-)}{}^2 = 0. \notag
\end{align}
Therefore for every $B \in \borel(\clint{0,T})$ we find
\begin{align}
  & \D\!\ (\norm{y_1(\cdot) - y_2(\cdot)}{}^2)(B) \notag \\
    & = \D\!\ (\norm{y_1(\cdot) - y_2(\cdot)}{}^2)(B \cap \cont(u)) + 
            \D\!\ (\norm{y_1(\cdot) - y_2(\cdot)}{}^2)(B \cap \discont(u)) \notag \\
    & = \int_{B \cap \cont(u)} \de \D\!\ (\norm{y_1(\cdot) - y_2(\cdot)}{}^2) + 
          \sum_{t \in B \cap \discont(u)} \D\!\ (\norm{y_1(\cdot) - y_2(\cdot)}{}^2)(\{t\}) \le 0 \notag
\end{align}
which implies that $t \longmapsto \norm{y_1(t-) - y_2(t-)}{}^2$ is nonincreasing and leads to the uniqueness of the solution since $y_i(t) = \Proj_{\C(t)}(y_i(t-))$ and $y_i(t+) = \Proj_{\C(t+)}(y_i(t-))$ for $i=1,2$. 
\end{proof}


\section{The 1-Lipschitz case}

Our proof of the main Theorem \ref{main thm} is based on the fact that problem \eqref{constr. arbBV sweep}--\eqref{i.c. arbBV sweep} is reduced to the case when $\C$ is $1$-Lipschitz continuous with respect to $\e$, i.e. 
\begin{equation}\label{e 1-Lip}
  \e(\C(t),\C(s)) \le s - t \qquad \forall t, s \in \ar\ :\ 0 \le t \le s.
\end{equation}
To be more precise the following theorem holds:

\begin{Thm}\label{Lip theorem}
Assume that $y_0 \in \H$ and that $\C : \clsxint{0,\infty} \function \Conv_\H$ is $1$-Lipschitz continuous with respect to $\e$, i.e. \eqref{e 1-Lip} holds. Then exists a unique function $y \in \Lip_\loc(\clsxint{0,\infty};\H)$ such that 
\begin{align}
  & y(t) \in \C(t), \label{constr. arbBV sweep-lip} \\
  & -y'(t) \in N_{\C(t)}(y(t)) \qquad \text{for $\leb^1$-a.e. $t \ge 0$}, \label{diff inclu-lip}\\
  & y(0) = \Proj_{\C(0)}(y_0). \label{i.c. arbBV sweep-lip}
\end{align}
We will call this unique solution $\Sw(\C,y_0)$ so that we define  
$\S$ $:$ $\Lip_\loc(\clsxint{0,\infty};\Conv_\H) \times \H \function \Lip_\loc(\clsxint{0,\infty};\H)$, the solution operator associating with every 
pair $(\C,y_0)$ the only $y$ satisfying \eqref{constr. arbBV sweep-r}--\eqref{i.c. arbBV sweep-r}.
\end{Thm}

This result was proved in \cite{Mor77} as a consequence of the more general formulation 
\eqref{constr. arbBV sweep-r}--\eqref{i.c. arbBV sweep-r}. Since we want to follow the opposite direction, for the sake of completeness and in order to be independent of the proofs of \cite{Mor77}, we provide here a direct proof.
As far as we know this proof has an element of novelty based on the fact that we exploit an integral formulation of 
\eqref{diff inclu-lip} in order to pass to the limit in the approximation procedure. Let us start by proving an integral formulation for the more general inclusion \eqref{diff inclu}.

\begin{Lem}\label{L:BV-int. form.}
Let us assume that $\mu : \borel(\clsxint{0,\infty}) \function \clint{0,\infty}$ is a Borel measure and that
$\C \in \BR^\r_\loc(\clsxint{0,\infty};\Conv_\H)$. If $v \in \L^1_\loc(\mu;\H)$, $y \in \BV_\loc^\r(\clsxint{0,\infty};\H)$ and $y(t) \in \C(t)$ for every 
$t \in \clsxint{0,\infty}$, then the following two conditions are equivalent.
\begin{itemize}
\item[(i)]
  $v(t) \in N_{\C(t)}(y(t))$ for $\mu$-a.e. in $\clsxint{0,T}$, for every $T > 0$.
\item[(ii)]
  $\displaystyle{\int_{\clsxint{0,T}} \duality{y(t) - z(t)}{v(t)} \de \mu(t) \le 0}$ for every $T > 0$ and every bounded 
  $\mu$-measurable function $z : \clsxint{0,T} \function \H$ such that $z(t) \in \C(t)$ for every $t \in \clsxint{0,T}$.
\end{itemize}
\end{Lem}

\begin{proof}
Let us start by assuming that (i) holds and let $z : \clsxint{0,T} \function \H$ be a bounded $\mu$-measurable function such that 
$z(t) \in \C(t)$ for every $t \in \clsxint{0,T}$. Then it follows that
\[
  \duality{y(t) - z(t)}{v(t)} \le 0 \qquad \text{for $\mu$-a.e. $t \in \clsxint{0,T}$},
\]
and integrating over $\clsxint{0,T}$ we infer condition (ii). Now assume that (ii) is satisfied and recall that if  
$f \in \L^1(\mu,\clsxint{0,T};\H)$ then there exists a $\mu$-zero measure set $Z$ such that $f(\clsxint{0,T} \setmeno Z)$ is separable (see, e.g., \cite[Property M11, p. 124]{Lan93}), therefore from (the proof) of 
\cite[Corollary 2.9.9., p. 156]{Fed69} it follows that 
\begin{equation}\label{Dl-Lebesgue points}
  \lim_{h \searrow 0} \frac{1}{\mu(\clint{t,t+h})} \int_{\clint{t,t+h}} \norm{f(\tau) - f(t)}{E} \de \mu(\tau)
  = 0 \qquad \text{for $\mu$-a.e. $t \in \clsxint{0,T}$}.
\end{equation}
The points $s$ satisfying \eqref{Dl-Lebesgue points} are called \emph{right $\mu$-Lebesgue points of $f$}. Let $L$ be the set of right $\mu$-Lebesgue points for $\tau \longmapsto \blu{v}(\tau)$, fix $t \in L$, and choose $\zeta_t \in \C(t)$ arbitrarily. Since $y \in \BV_\loc^\r(\clsxint{0,\infty};\H)$ we have that $t$ is a right $\mu$-Lebesgue point of $y$, therefore
\begin{align}
    &  \frac{1}{\mu(\clint{t,t+h})} \int_{\clint{t,t+h}} \left|\duality{y(\tau)}{\blu{v}(\tau)} - \duality{y(t)}{v(t)}\right| \de \mu(\tau)
      \notag \\
     & \le
      \frac{1}{\mu(\clint{t,t+h})} \int_{\clint{t,t+h}}
        \left(\norm{y}{\infty} \norm{v(\tau) - v(t)}{} + \norm{\blu{v}}{\infty} \norm{y(\tau) - y(t)}{}\right) \de \mu(\tau), \label{pto leb 1}
\end{align}
thus
\begin{equation}
\lim_{h \searrow 0} \frac{1}{\mu(\clint{t,t+h})} \int_{\clint{t,t+h}} \duality{y(\tau)}{v(\tau)}  \de \mu(\tau)
  = \duality{y(t)}{v(t)}, \label{pto leb 2}
\end{equation}
i.e. $t$ is a right $\mu$-Lebesgue point of $\tau \longmapsto \duality{y(\tau)}{v(\tau)}$.
If $\zeta : \clsxint{0,T} \function \H$ is defined by $\zeta(\tau) := \Proj_{\C(\tau)}(\zeta_t)$, $\tau \in \clsxint{0,T}$, we have that $\zeta$ is right continuous (see \cite[Proposition 2e, Remark 1]{Mor77}, therefore
$\zeta$ is bounded and measurable, $\zeta(t) = \zeta_t$, and arguing as before we see that $t$ is a right 
$\mu$-Lebesgue point of $\tau \longmapsto \duality{\zeta(\tau)}{v(\tau)}$. Since $\zeta(\tau) \in \C(\tau)$ for every 
$\tau \in \clsxint{0,T}$, the function
$z(\tau) := \indicator_{\clsxint{0,T} \cap \clint{t,t+h}}(\tau) \zeta(\tau) + 
\indicator_{\clsxint{0,T} \setmeno \clint{t,t+h}}(\tau) y(\tau)$, $\tau \in \clsxint{0,T}$, is well defined for every sufficiently small $h > 0$ and $z(\tau) \in \C(\tau)$ for every $\tau \in \clsxint{0,T}$, thus we can take $z$ in condition (ii) and we get
\[
  \int_{\clint{t,t+h}} \duality{y(\tau)}{v(\tau)}\de \mu(\tau) \le 
  \int_{\clint{t,t+h}} \duality{\zeta(\tau)}{v(\tau)} \de \mu(\tau).
\] 
Dividing this inequality by $\mu(\clint{t,t+h})$ and taking the limit as $h \searrow 0$ we get
$\duality{y(t) - \zeta_t}{v(t)} \le 0$. Therefore, as $\mu(L)$ = 0, we have proved that
\begin{equation}
  \duality{y(t) - \zeta}{v(t)} \le 0 \quad \forall \zeta \in \C(t), \quad \text{for $\mu$-a.e. $t \in \clsxint{0,T}$,} \notag
\end{equation}
i.e. condition (i) holds.
\end{proof}

As a corollary we get the desired integral formulation for \eqref{diff inclu-lip}.

\begin{Cor}\label{L:Lip-int. form.}
Let $\C : \clsxint{0,\infty} \function \Conv_\H$ be such that \eqref{e 1-Lip} holds. If $y \in \Lip_\loc(\clsxint{0,\infty};\H)$ and $y(t) \in \C(t)$ for every $t \in \clsxint{0,\infty}$, then the following two conditions are equivalent.
\begin{itemize}
\item[(i)]
  $-y'(t) \in N_{\C(t)}(y(t))$ for $\leb^1$-a.e. in $\clint{0,T}$, for every $T > 0$.
\item[(ii)]
  $\displaystyle{\int_0^T \duality{y(t) - z(t)}{y'(t)} \de t \le 0}$ for every $T > 0$ and every bounded 
  $\leb^1$-measurable function $z : \clint{0,T} \function \H$ such that $z(t) \in \C(t)$ for every $t \in \clint{0,T}$.
\end{itemize}
\end{Cor}

\begin{proof}
We can apply Lemma \eqref{L:BV-int. form.} with $\mu = \leb^1$ and $v = -y'$.
\end{proof}


Now we provide the proof of Theorem \ref{Lip theorem} following an implicit discretization scheme. The proof is obviously inspired by the paper \cite{Mor77} and a crucial fact for getting proper a priori estimates is the 
following inequality holding for $\A, \B \in \Conv_\H$ and $x, y \in \H$ (see \cite[Lemma 1(2a)]{Mor77}):
\begin{equation}\label{estimate for projections}
  \norm{\Proj_{\A}(x) - \Proj_{\B}(y)}{}^2 - \norm{x-y}{}^2 \le
  2\d(x,\A)\e(\B,\A) + 2 \d(y,\B)\e(\A,\B)
\end{equation}

\begin{proof}[Proof of Theorem \ref{Lip theorem}]
First of all observe that the uniqueness of solutions can be inferred exactly as in Lemma \ref{uniqueness},
since \eqref{constr. arbBV sweep}--\eqref{i.c. arbBV sweep} hold with $\mu = \leb^1$ and $v = y'$. Concerning existence, for every $n \in \en \cup \{0\}$, set 
\begin{equation}\label{t_j^n = j/2^n}
  t_j^{n} := \frac{j}{2^n}, \quad j \in \en \cup \{0\}
\end{equation}
and we define recursively the sequence $(y_j^n)_{j=0}^\infty$ by the three following conditions:
\begin{align}
   & y_0^n := \Proj_{\C(0)}(y_0), \vspace{1ex}  \label{implicit scheme for Lip-sweeping processes1} \\
   & y_j^n \in \C(t_j^n), \qquad j \in\ \en.\vspace{1.1ex} \label{implicit scheme for Lip-sweeping processes2} \\
   & \dfrac{y_j^n -y_{j-1}^n}{t_j^n - t_{j-1}^n} \in -N_{\C(t_j^n)}(y_j^n), \label{implicit scheme for Lip-sweeping processes3}
      \qquad j \in\ \en.
\end{align}
Observe that \eqref{implicit scheme for Lip-sweeping processes1}-\eqref{implicit scheme for Lip-sweeping processes3} is an implicit time discretization scheme for \eqref{diff inclu-lip}--\eqref{i.c. arbBV sweep-lip}.
From the fact that $N_{\C(t_j^n)}(y_j^n)$ is a cone we infer that \eqref{implicit scheme for Lip-sweeping processes3} 
is equivalent to $y_{j-1}^n - y_j^n \in N_{\C(t_j^n)}(y_j^n)$, which can be rewritten as
\begin{equation}\label{y j n = ..}
  y_j^n = \Proj_{\C(t_j^n)}(y_{j-1}^n) = \Proj_{\C(\frac{j}{2^n})}(y_{j-1}^n) \qquad \forall j \in \en.
\end{equation}
Therefore using the $1$-Lipschitzianity of $\C$ with respect to $\e$, we get
\begin{align}\label{estimate a}
\norm{y_{1}^{n+1} - y_{0}^{n}}{}^2 
  & =   \norm{\Proj_{\C(t_1^{n+1})}(y_{0}^n) - y_{0}^n}{}^2 = \d(y_0^n,\C(t_1^{n+1}))^2\notag \\
  & \le   \e(\C(0),\C(t_1^{n+1}))^2 \le (t_1^{n+1})^2 = \frac{1}{2^{2n+2}}
\end{align}
and more generally, using also \eqref{estimate for projections}, for every $j \ge 1$ we have
\begin{align}\label{estimate b}
  \norm{y_{2j+1}^{n+1} - y_j^n}{}^2 
    &  =   \norm{\Proj_{\C(t_{2j+1}^{n+1})}(y_{2j}^{n+1}) - y_{j}^n}{}^2 \notag \\
    & \le \norm{y_{2j}^{n+1} - y_{j}^n}{}^2 + 
             2\d\left(y_{2j}^{n+1}, \C(t_{2j+1}^{n+1})\right)\d\left(y_j^n, \C(t_{2j+1}^{n+1})\right) \notag \\
    & \le \norm{y_{2j}^{n+1} - y_j^n}{}^2 + 
             2\e\left(\C(t_{2j}^{n+1}), \C(t_{2j+1}^{n+1})\right)\e\left(\C(t_{j}^{n}), \C(t_{2j+1}^{n+1})\right) \notag \\  
    & =  \norm{y_{2j}^{n+1} - y_j^n}{}^2 + 
             2\e\left(\C\left(\frac{2j}{2^{n+1}}\right), \C\left(\frac{2j+1}{2^{n+1}}\right)\right)
               \e\left(\C\left(\frac{2j}{2^{n+1}}\right), \C\left(\frac{2j+1}{2^{n+1}}\right)\right) \notag \\    
    & \le \norm{y_{2j}^{n+1} - y_j^n}{}^2 + 
             2\left(\frac{2j+1}{2^{n+1}}-\frac{2j}{2^{n+1}}\right)  \left(\frac{2j+1}{2^{n+1}}-\frac{2j+1}{2^{n+1}}\right)   \notag \\ 
    & = \norm{y_{2j}^{n+1} - y_j^n}{}^2 + 
             2\frac{1}{2^{n+1}} \frac{1}{2^{n+1}}  
    = \norm{y_{2j}^{n+1} - y_j^n}{}^2 + 
             \frac{2}{2^{2n+2}}.                         
\end{align}
Moreover for every $j \in \en$ the nonexpansivity of the projection yields
\begin{align}\label{estimate c}
\norm{y_{2j}^{n+1} - y_j^n}{}^2 
    & =   \norm{\Proj_{\C(\frac{j}{2^{n}})}(y_{2j-1}^{n+1}) - \Proj_{\C(\frac{j}{2^n})}(y_{j-1}^{n})}{}^2 \notag \\
    & \le \norm{y_{2j-1}^{n+1} - y_{j-1}^n}{}^2 \notag \\
    & =   \norm{y_{2(j-1)+1}^{n+1} - y_{j-1}^n}{}^2,        
\end{align}
thus, \eqref{estimate a}, \eqref{estimate b}, \eqref{estimate c}, and a recursive argument yield for $j \ge 1$
\begin{align}\label{estimate c}
  \norm{y_{2j+1}^{n+1} - y_j^n}{}^2  
   & \le \norm{y_{1}^{n+1} - y_{0}^n}{}^2 +  \frac{2j}{2^{2n+2}}  \notag \\
   & \le \frac{1}{2^{2n+2}} + \frac{2j}{2^{2n+2}}  = \frac{1+2j}{2^{2n+2}}
\end{align}
and for $j \ge 1$
\begin{align}\label{estimate d}
\norm{y_{2j}^{n+1} - y_j^n}{}^2 
    &  \le     \norm{y_{2(j-1)+1}^{n+1} - y_{j-1}^n}{}^2 \notag \\
    & \le \frac{1+2(j-1)}{2^{2n+2}} = \frac{2j-1}{2^{2n+2}}.     
\end{align}
We now define the step function $y_n : \clsxint{0,\infty} \function \H$ by
\begin{equation}
  y_n(t) := y_{j}^n \qquad \text{if $t \in \clsxint{t_{j}^{n}, t_{j+1}^{n}}$}, \quad j \in \en \cup \{0\}.
\end{equation}
and we are going to prove that the sequence $(y_n)$ is locally uniformly Cauchy. To this aim we fix $T > 0$, 
$t \in \clint{0,T}$, and $n \in \en$. If $k, K \in \en \cup \{0\}$ are the unique integers such that
\begin{equation}\label{k and K}
\frac{k}{2^n} = t_k^n \le t < t_{k+1}^n =\frac{k+1}{2^n}, \qquad
\frac{K}{2^n}  \le T < \frac{K+1}{2^n},
\end{equation}
we have two possibilities: either
\begin{equation}\label{Cauchy estimate 1st case}
  t_k^n = \frac{k}{2^n} = \frac{2k}{2^{n+1}} = t_{2k}^{n+1}  \le t < \frac{2k+1}{2^{n+1}} = t_{2k+1}^{n+1}
\end{equation}
or  
\begin{equation}\label{Cauchy estimate 2nd case}
  \frac{2k+1}{2^{n+1}} = t_{2k+1}^{n+1} \le t < t_{2k+2}^{n+1} = \frac{2k+2}{2^{n+1}} = \frac{k+1}{2^n}.
\end{equation}
In the first case \eqref{Cauchy estimate 1st case}, from \eqref{estimate d} we get that
\begin{equation}
  \norm{y_{n+1}(t) - y_n(t)}{}^2 
     =   \norm{y_{2k}^{n+1} - y_k^n}{}^2  
     \le \frac{2k-1}{2^{2n+2}},
 \end{equation}
in the second case \eqref{Cauchy estimate 2nd case}, from \eqref{estimate c} we infer that
\begin{equation}
  \norm{y_{n+1}(t) - y_n(t)}{}^2 
     =   \norm{y_{2k+1}^{n+1} - y_k^n}{}^2 \le \frac{1+2k}{2^{2n+2}}
 \end{equation}
Therefore in every case thanks to \eqref{k and K} we find that
\begin{equation}
\norm{y_{n+1}(t) - y_n(t)}{}^2  \le \frac{1+2k}{2^{2n+2}} \le \frac{1+2K}{2^{2n+2}} \le \frac{1+2^{n+1}T}{2^{2n+2}}
\end{equation}
thus
\begin{equation}
  \sum_{n=0}^\infty \norm{y_{n+1}(t) - y_n(t)}{} \le \sum_{n=0}^\infty \frac{(1+2^{n+1}T)^{1/2}}{2^{n+1}}
   < \infty,
\end{equation}
and we have that $(y_n)$ is uniformly Cauchy on $\clint{0,T}$. It follows that there exists a right continuous function
$y : \clsxint{0,\infty} \function \H$ such that 
\begin{equation}
y_n \to y \qquad \text{uniformly on $\clint{0,T}, \quad \forall T > 0$}.
\end{equation} 
In particular it follows, by the closedness of $\C(t)$, that 
\begin{equation}
  y(t) \in \C(t) \qquad \forall t \ge 0.
\end{equation}
Now take $t, s \in \ar$ with $0 \le t < s$. Then there exist
$j_n, k_n \in \en$ such that $j_n \le k_n$ and
\begin{equation}
  \frac{j_n}{2^n} \le t < \frac{j_n+1}{2^n}, \qquad \frac{k_n}{2^n} \le s < \frac{k_n+1}{2^n},
\end{equation}
and we have
\begin{align}
   \norm{y_n(t) - y_n(s)}{}  
  & = \norm{y_{j_n}^n - y_{k_n}^n}{} \notag \\
  & \le \norm{y_{j_n}^n - y_{j_n+1}^n}{} + \cdots + \norm{y_{k_n - 1}^n - y_{k_n}^n}{} \notag \\
  & \le \norm{y_{j_n}^n - \Proj_{\C(t_{j_n+1}^n)}(y_{j_n}^n)}{} + \cdots + 
          \norm{y_{k_n - 1}^n - \Proj_{\C(t_{k_n}^n)}(y_{k_n-1}^n)}{} \notag \\ 
  & = \d(y_{j_n}^n,\C(t_{j_n+1}^n)) + \cdots + 
          \d(y_{k_n - 1}^n , \C(t_{k_n}^n)) \notag \\      
  & \le \e(\C(t_{j_n}^n),\C(t_{j_n+1}^n)) + \cdots + 
          \e(\C(t_{k_n - 1}^n) , \C(t_{k_n}^n)) \notag \\          
  & \le (t_{j_n+1}^n - t_{j_n}^n) + (t_{j_n+2}^n - t_{j_n+1}^n) + \cdots + (t_{k_n}^n - t_{k_n - 1}^n) \notag \\               
  & =   \frac{1}{n} + t_{k_n}^n - t_{j_n+1}^n \notag \\
  & \le  \frac{1}{n}+ s - t \notag
\end{align}
thus
\begin{equation}
  \norm{y(t) - y(s)}{} 
     = \lim_{n \to \infty} \norm{y_n(t) - y_n(s)}{} 
   \le   \lim_{n \to \infty} \frac{1}{n} + s - t = |t-s| \notag
\end{equation}
and we have found that $y$ is a $1$-Lipschitz continuous function. Now let $x_n$ be the piecewise affine interpolant of $y_n$:
\begin{equation}
  x_n(t) := y_{j-1}^n + \frac{t-t^n_{j-1}}{t_j^n - t_{j-1}^n}(y_j^n - y_{j-1}^n), 
  \qquad \text{if $t \in \clsxint{t_{j-1}^n, t_j^n}$}.
\end{equation}
It is immediately seen that
\begin{equation}
  x_n \to y \qquad \text{uniformly on $\clint{0,T}$},
\end{equation}
moreover we have 
\begin{equation}
  x_n'(t) = \frac{y_j^n - y_{j-1}^n}{t_j^n - t_{j-1}^n} \in -N_{C(t_j^n)(y_j^n)} \qquad \forall t \in \opint{t_{j-1}^n, t_j^n}
\end{equation}
and
\begin{equation}
  \norm{x_n'(t)}{} = \frac{\norm{y_j^n - y_{j-1}^n}{}}{t_j^n - t_{j-1}^n} \le 1 \qquad \forall t \neq t_j^n,
\end{equation}
therefore $\norm{x_n'}{\L^2(\clint{0,T};\H)}$ is bounded and, at least for a subsequence which we do not relabel, we have that 
\begin{equation}
  x_n' \to y' \qquad \text{weakly in $\L^2(\clint{0,T};\H)$}.
\end{equation}
Thus, if $z : \clint{0,T} \function \H$ is a bounded measurable function such that $z(t) \in \C(t)$ for every $t \in \clint{0,T}$,
we have that
\begin{align}
   \int_0^T \duality{y(t) - z(t)}{y'(t)} \de t 
     & = \lim_{n \to \infty} \int_0^T \duality{y_n(t) - z(t)}{x_n'(t)} \de t \le 0,
\end{align} 
hence thanks to the integral characterization of Corollary \ref{L:Lip-int. form.} we have proved the theorem. 
\end{proof}


\section{Geodesics for the retraction}

In this section we introduce the class of geodesics with respect to the excess $\e$ which allow us to reduce the $\BR$
sweeping processes to the Lipschitz continuous case. 

\begin{Def}\label{D:F}
Assume that $\A, \B \in \Conv_\H$ and set $\rho := \e(\A,\B)$. 
We define the curve $\F_{\A,\B} : \clint{0,1} \function \Conv_\H$ by 
\begin{equation}\label{catching-up geodesic-2}
  \F_{(\A,\B)}(t) := 
  \begin{cases}
    \A & \text{if $t=0$} \\
    \B + (1-t)D_\rho = \B + D_{(1-t)\rho} & \text{if $0 < t \le 1$}
  \end{cases}
\end{equation}
\end{Def}

\begin{Prop}
If $\A, \B \in \Conv_\H$ and $\F_{(\A,\B)} : \clint{0,1} \function \Conv_\H$ is defined as in Definition \ref{D:F} we have
\begin{equation}\label{F geodesic}
  \e(\F(s),\F(t)) = (t-s)\e(\A,\B) \qquad \forall s, t \in \clint{0,1},\ s < t,
\end{equation}
in particular $\Lipcost(\F) = (t-s)$ and we can call $\F$ a $\e$-geodesic connecting $\A$ to $\B$.
\end{Prop}

\begin{proof}
For every $t > 0$ we have $\F(0) = \A \subseteq \B + D_{\rho} = \A + \B_{(1-t)\rho} + \B_{t\rho}$ thus
$\e(\F(0),\F(t)) \le t\rho$. If $0 < s \le t$ we have
\begin{align}
  \F(s) 
    & = \B + D_{(1-s)\rho} \notag \\
    & = \B + D_{(1- t)\rho} + D_{(t-s)\rho} \notag \\   
    & = \F(t) + D_{(t-s)\rho}.
\end{align}
Therefore $\e(\F(s),\F(t)) \le (t - s)\rho = (t-s)\e(\A,\B)$. On the other hand we have, for instance,
$\e(\A,\B) \le \e(\A,\F(s)) + \e(\F(s),\F(t)) + \e(\F(t),\B) \le s\e(\A,\B) + \e(\F(s),\F(t)) + (1-t)\e(\A,\B)$, hence
$(t-s)\e(\A,\B) \le \e(\F(s),\F(t))$ and \eqref{F geodesic} is proved.
\end{proof}

The next Lemma shows that the solutions of the sweeping processes driven by $\F_{(\A,\B)}$ are always straight
line segments connecting the initial datum $y_0$ to its projection to $\B$.

\begin{Lem}\label{L:particular sweeping process}
Let $\A, \B \in \Conv_\H$ be such that $\rho := \e_\hausd(\A,\B) < \infty$ and let 
$\F_{(\A, \B)} = \F : \clint{0,1} \function \Conv_\H$ be defined by \eqref{catching-up geodesic-2}. If $y_0 \in \A$, then let $t_0 \in \clint{0,1}$ be the unique number such that  
\begin{equation}\label{t_0 s.t. d(y_0,B) = (1-t_0)delta}
  \norm{y_0 - \Proj_\B(y_0)}{} = (1-t_0)\rho
\end{equation}
and define $y \in \Lip(\clint{0,1};\H)$ by
\begin{equation}\label{solution for catching-up geodesic-2}
  y(t) :=
  \begin{cases}
    y_0 & \text{if $t \in \clsxint{0,t_0}$} \\
    y_0 + \dfrac{t-t_0}{1-t_0}(\Proj_{\B}(y_0) - y_0) & \text{if $t_0 \neq 1$, $t \in \clsxint{t_0,1}$} \\
    \Proj_\B(y_0) & \text{if $t = 1$}
  \end{cases}.
\end{equation}
Then
\begin{alignat}{3} 
  & y(t) \in \F_{(\A, \B)} (t) & \qquad & \forall t \in \clint{0,1}, 
     \label{y in G - Lip} \\
  & y'(t) + N_{\F_{(\A, \B)} (t)}(y(t)) \ni 0 & \qquad &  \text{for $\leb^{1}$-a.e. $t \in \clint{0,1}$}, 
      \label{diff. incl. G - Lip} \\
  & y(0) = \Proj_{\F_{(\A, \B)}(0)}(y_{0}) = y_0,
      \label{in. cond. G - Lip}
\end{alignat}
i.e. $y$ is the unique solution of the sweeping process driven by $\F_{(\A,\B)}$ with initial condition $y_0 \in \A$. 
\end{Lem}

\begin{proof}
We use the notation $\K_\rho := \K + D_\rho$ for $\K \in \Conv_\H$ and $\rho \ge 0$. 
If $y_0 \in \B$ we have that $t_0 = 1$, $y(t) = y_0 = \Proj_\B(y_0) \in \G(t)$ and $y'(t) = 0$ for every $t \in \clint{0,1}$, and we are done. Therefore we assume that $y_0 \not\in \B$, i.e. $t_0 < 1$, thus from 
\cite[Lemma 4.1-(ii)]{Rec16a} and formula \eqref{t_0 s.t. d(y_0,B) = (1-t_0)delta} we deduce that
\begin{equation}\label{y_0 in bordo B_(1-t)}
  y_0 \in \partial \B_{(1-t)\rho} \iff t = t_0,
\end{equation}
that is $t_0$ is the first time when $\partial \B_{(1-t)\rho}$ meets $y_0$. If $t \in \cldxint{t_0,1}$ we have that 
\[
  \norm{y(t) - \Proj_\B(y_0)}{} = \frac{1-t}{1-t_0} \norm{y_0 - \Proj_\B(y_0)}{} = \frac{1-t}{1-t_0}(1-t_0)\rho = (1-t)\rho,
\]
therefore
\begin{equation}\label{y(t) in bordo(B_((1-t)delta)) for t > t_0}
  y(t) \in \partial \B_{(1-t)\rho} \qquad \forall t \in \clint{t_0,1}.
\end{equation}
Therefore, since
\[
  y'(t) =
  \begin{cases}
    0 & \text{if $t \in \opint{0,t_0}$} \\
    \dfrac{1}{1-t_0}(\Proj_{\B}(y_0) - y_0) & \text{if $t \in \opint{t_0,1}$}
  \end{cases},
\]
we infer from \cite[Lemma 4.1-(iii)]{Rec16a} that 
$y'(t) \in N_{\F(t)}(y(t))$ for every $t \in \opint{0,1} \setmeno \{t_0\}$.
\end{proof}


\section{Sweeping processes with arbitrary locally bounded retraction}\label{proofs}

In this section we provide the proofs of the main theorems.

\begin{proof}[Proof of Theorem \ref{main thm}]
Now we assume that $\C$ has bounded local retraction and we recall that 
$\ell_\C : \clsxint{0,\infty} \function \clsxint{0,\infty}$ is defined by
\begin{equation}
  \ell_\C(t) := \ret(\C;\clint{0,t}), \qquad t \ge 0.
\end{equation}
The function $\ell_\C$ is increasing, therefore $\ell_\C^{-1}(\tau)$ is always a (possibly degenerate) interval for every $\tau \in \ell_\C(\clsxint{0,\infty})$ and we can define $\Ctilde : \ell_\C(\clsxint{0,\infty}) \function \Conv_\H$ in the following way: 
\[
  \Ctilde(\tau) :=
    \begin{cases}
      \C(t)& \text{if $\ell_\C^{-1}(\tau)$ is a singleton and $\ell_\C^{-1}(\tau) = \{t\}$} \\
      \C(t+) & \text{if $\ell_\C^{-1}(\tau)$ is not a singleton and $\inf\ell_\C^{-1}(\tau) = t$}.
    \end{cases}
\]
Let us observe that if $\ell_\C$ is not constant on any nondegenerate interval, then one has $\C(t) = \Ctilde(\ell_\C(t))$ 
for every $t \ge 0$, so that $\Ctilde$ can be considered as a reparametrization of $\C$ by the ``arc length'' $\ell_\C$. Let us also observe that $\C$ is set-theoretically increasing in time on the intervals where 
$\ell_\C$ is constant, therefore the solution of the sweeping process driven by $\C$ is expected to be constant on these intervals.
We now claim that $\Ctilde$ is $1$-Lipschitz continuous with respect to $\e$, and to this aim we take $\sigma, \tau \in \ell_\C(\clsxint{0,\infty})$ with $\sigma < \tau$ and assume that 
$s = \inf\ell_\C^{-1}(\sigma) \le \sup\ell_\C^{-1}(\sigma) = s^*$, $t = \inf\ell_\C^{-1}(\tau) \le \sup\ell_\C^{-1}(\tau) = t^*$ for some $s, s^*, t, t^* \ge 0$. If $\ell_\C^{-1}(\sigma)$ and $\ell_\C^{-1}(\tau)$ are both singletons, then 
$\e(\Ctilde(\sigma),\Ctilde(\tau))$ $= \e(\C(s),\C(t)) \le \ret(\C;\clint{s,t}) = \tau - \sigma$. If $\ell_\C^{-1}(\sigma)$ is not a singleton and $\ell_\C^{-1}(\tau)$ is a singleton, then $\e(\Ctilde(\sigma),\Ctilde(\tau)) = \e(\C(s+),\C(t)) = 
\e(\C((s+s^*)/2),\C(t)) \le \ret(\C;\clint{(s+s^*)/2,t}) = \tau - \sigma$. If instead $\ell_\C^{-1}(\sigma)$ is a singleton and 
$\ell_\C^{-1}(\tau)$ is not a singleton, then $\e(\Ctilde(\sigma),\Ctilde(\tau)) = \e(\C(s),\C(t+)) = \e(\C(s),\C((t+t^*)/2)) \le \ret(\C;\clint{s,(t+t^*)/2}) = \tau - \sigma$. Finally if both $\ell_\C^{-1}(\sigma)$ and $\ell_\C^{-1}(\tau)$ are not singletons, then $\e(\Ctilde(\sigma),\Ctilde(\tau)) = \e(\C(s+),\C(t+)) = \e(\C((s+s^*)/2),\C((t+t^*)/2)) \le 
\ret(\C;\clint{(s+s^*)/2,(t+t^*)/2}) = \tau - \sigma$, and we have proved that $\Ctilde$ is $1$-Lipschitz continuous with respect to $\e$. Now we extend the definition
of $\Ctilde$ over the whole $\clsxint{0,\infty}$ by setting 
\begin{alignat}{3}
  & \Ctilde(\sigma) := \F_{(\C(t-),\C(t))}\left(\frac{\sigma - \ell_\C(t-)}{\ell_\C(t) - \ell_\C(t-)}\right) & \qquad
      & \text{if $\sigma \in \opint{\ell_\C(t-), \ell_\C(t)}$, if $\ell_\C(t-) \neq \ell_\C(t)$}, \\
  & \Ctilde(\ell_\C(t-)) := \C(t-) & \qquad 
       & \text{if $\ell_\C(t-) \not\in \ell_\C(\clsxint{0,\infty})$, if $\ell_\C(t-) \neq \ell_\C(t)$}, 
\end{alignat}
and 
\begin{alignat}{3}
   & \Ctilde(\sigma) := \F_{(\C(t),\C(t+))}\left(\frac{\sigma - \ell_\C(t)}{\ell_\C(t+) - \ell_\C(t)}\right) & \qquad
      & \text{if $\sigma \in \cldxint{\ell_\C(t), \ell_\C(t+)}$, if $\ell_\C(t) \neq \ell_\C(t+)$},
\end{alignat}
thus the resulting curve $\Ctilde : \clsxint{0,\infty} \function \Conv_\H$ is a $1$-Lipschitz continuous function 
with respect to $\e$ because  $\F_{(\C(t-),\C(t))}$ and  $\F_{(\C(t),\C(t+))}$ are geodesics connecting respectively
$\C(t-)$ to $\C(t)$ and $\C(t)$ to $\C(t+)$, 
and because $\C(\inf \ell^{-1}(\ell_\C(t-))+) \subseteq \C(t-)$ if $\ell_\C(t-) \in \ell_\C(\clsxint{0,\infty})$.
Therefore we have that $\Sw(\Ctilde) \in \Lip_\loc(\clsxint{0,\infty};\H)$ where $\Sw$
is the solution operator of the Lipschitz sweeping process defined in Theorem \ref{Lip theorem}. Let us set
\begin{equation}\notag
   \yhat := \Sw(\Ctilde, y_0), \qquad y := \Sw(\Ctilde, y_0) \circ \ell_{\C},
\end{equation}
and let us prove that $y$ solves \eqref{constr. arbBV sweep}-\eqref{i.c. arbBV sweep}.
It is obvious that $y(t) \in \C(t)$ when $\ell_\C^{-1}(\ell_{\C}(t)) = \{t\}$. If instead 
$\ell_\C^{-1}(\ell_{\C}(t))$ is not a singleton, we have that 
$y(t) = \yhat(\ell_\C(t)) \in \Ctilde(\ell_\C(t)) = \C(\inf\ell_\C^{-1}(\ell_\C(t))+) \subseteq \C(t)$, thus condition 
\eqref{constr. arbBV sweep} is satisfied. It is very easy to check that \eqref{i.c. arbBV sweep} holds true.
Since $\yhat$ is Lipschitz continuous and $\ell_\C$ is increasing, it is clear that $y \in \BV(\clint{0,T};\H)$, and that $y$ is left continuous (respectively: right continuous) if and only if $\ell_\C$ is left continuous (respectively: right continuous), so that $\discont(y) = \discont(\ell_\C) = \discont(\C)$.
If $v : \clsxint{0,\infty} \function \H$ is defined by
\begin{equation}\label{density w}
  v(t) := 
  \begin{cases}
  \yhat' (\ell_{\C}(t)) & \text{if $t \in \cont(\ell_\C)$} \\
  \ \\
  \dfrac{\yhat(\ell_\C(t+)) - \yhat(\ell_\C(t-))}{\ell_\C(t-) - \ell_\C(t-)} & \text{if $t \in \discont(\ell_\C)$}
  \end{cases},
\end{equation}
then, since $\ell_\C$ is increasing, from Proposition \ref{P:BV chain rule}-(ii) we infer that 
$y  \in \BV_\loc(\clsxint{0,\infty};\H)$ and $\D y = v \D\ell_\C$, i.e. \eqref{Dy = vmu} holds with $\mu = \D \ell_\C$.
Let us set
\begin{equation}\label{Z}
  Z := \{t \in \clsxint{0,\infty}\ :\ -\yhat'(t) \not\in N_{\Ctilde(t)}(\yhat(t))\}.
\end{equation} 
From formula \eqref{diff inclu-lip} we deduce that
\[
  \leb^{1}(Z) = 0,
\]
therefore, thanks to Proposition \ref{P:BV chain rule}-(i), we have that
\begin{align}
  & \D\ell_{\C}(\{t \in \cont(\ell_{\C})\ :\ -v(t) \not\in N_{\C(t)}(y(t))\}) \notag \\
  = & \D\ell_{\C}
         (\{t \in \cont(\ell_\C)\ :\ -\yhat'(\ell_{\C}(t)) \not\in N_{\Ctilde(\ell_{\C}(t))}(\yhat(\ell_{\C}(t))\})  
          \notag \\ 
  = & \D\ell_{\C}(\{t \in \cont(\ell_{\C})\ :\ \ell_{\C}(t) \in Z\}) = \leb^{1}(Z) = 0, \label{Dl_C(...)}
\end{align}
hence \eqref{diff inclu} also holds with $\mu = \D \ell_\C$.
Now let us fix $t \in \discont(\ell_\C)$ and observe that if $\sigma \in \opint{\ell_\C(t-), \ell_\C(t)}$ then
$\Ctilde(\sigma) = \F_{(\C(t-),\C(t))}(\sigma-\ell_\C(t-)/(\ell_\C(t) - \ell_\C(t-)))$, thus by the semigroup property 
of \eqref{constr. arbBV sweep-lip}--\eqref{i.c. arbBV sweep-lip} and by Lemma \ref{L:particular sweeping process} we have 
\begin{align}
  v(t) 
   & = \yhat(\ell_\C(t)) \notag \\
   & = \Sw(\Ctilde, y_0)(\ell_\C(t)) \notag \\
   & = \Sw(\Ctilde(\cdot+\ell_\C(t-), \Sw(\Ctilde,y_0)(\ell_\C(t-))))(\ell_\C(t) - \ell_\C(t-)) \notag \\
   & = \Sw(\F_{(\C(t-),\C(t))}, \yhat(\ell_\C(t-)))(1) \notag \\
   & = \Sw(\F_{(\C(t-),\C(t))}, y(t-))(1) \notag \\
   & = \Proj_{\C(t)}(y(t-)).
\end{align}
On the other hand if $\sigma \in \opint{\ell_\C(t), \ell_\C(t+)}$ then
$\Ctilde(\sigma) = \F_{(\C(t),\C(t+))}(\sigma-\ell_\C(t)/(\ell_\C(t+) - \ell_\C(t)))$, therefore
\begin{align}
  v(t+) 
   & = \yhat(\ell_\C(t+)) \notag \\
   & = \Sw(\Ctilde, y_0)(\ell_\C(t)) \notag \\
   & = \Sw(\Ctilde(\cdot+\ell_\C(t), \Sw(\Ctilde,y_0)(\ell_\C(t))))(\ell_\C(t+) - \ell_\C(t)) \notag \\
   & = \Sw(\F_{(\C(t),\C(t+))}, \yhat(\ell_\C(t)))(1) \notag \\
   & = \Sw(\F_{(\C(t),\C(t+))}, y(t))(1) \notag \\
   & = \Proj_{\C(t+)}(y(t)),
\end{align}
thus \eqref{jump cond} is satisfied and we are done since uniqueness was proved in Section \ref{Uniqueness}.
\end{proof}

Now we can give the  

\begin{proof}[Proof of Theorem \ref{main-right continuous case}]
We know from Theorem \ref{main thm} that $y$ is right continuous, thus we only have to prove formula 
\eqref{diff inclu-r}. If $t \in \discont(u)$ then the first condition of \eqref{jump cond} reads 
\begin{equation}
  \duality{y(t)-\zeta}{y(t) - y(t-)} \le 0 \qquad \forall \zeta \in \C(t+).
\end{equation}
Hence, since $\cont(y) = \cont(\ell_\C)$ and $\ell_\C$ is right continuous, for every $z \in \L^\infty(\clsxint{0,T};\H)$ with 
$z(\clsxint{0,T}) \subseteq \Z$ we have
\begin{align}
  & \int_{\clsxint{0,T}} \duality{y(t)-z(t)}{\de\D y(t)} \notag \\
  & = \int_{\cont(u)\cap\clsxint{0,T}} \duality{y(t)-z(t)}{\de\D y(t)} \notag \\
  & \phantom{= \ } + 
          \sum_{t \in \discont(u)\cap\clsxint{0,T}} \duality{y(t)-z(t)}{y(t) - y(t-)} \le 0 \notag
\end{align}
and we can conclude invoking 
Lemma \ref{L:BV-int. form.}. 
\end{proof}

The proof of Theorem \ref{main-left continuous case} is analogous.





\end{document}